\documentclass{amsart}

\usepackage{amsmath}
\usepackage{amsxtra}
\usepackage{amscd}
\usepackage{amsthm}
\usepackage{amsfonts}
\usepackage{amssymb}
\usepackage{mathtools}
\usepackage{eucal}
\usepackage{graphicx}
\usepackage{mathrsfs}
\usepackage{color}
\usepackage{caption}
\usepackage[margin=1in]{geometry}
\usepackage{enumerate}
\usepackage{faktor}
\usepackage{bbm}
\usepackage{hyperref}
\usepackage[T1]{fontenc} % improved font encoding
\usepackage[utf8]{inputenc} % for better handling of non-ASCII characters

\hypersetup{
  colorlinks = true,
  urlcolor = cyan,
}
\usepackage{comment}

\theoremstyle{plain}
\newtheorem{theorem}{Theorem}
\newtheorem{lemma}[theorem]{Lemma}

\newtheorem{problem}[theorem]{Problem}

\theoremstyle{definition}
\newtheorem{definition}[theorem]{Definition}

\theoremstyle{remark}
\newtheorem{remark}[theorem]{Remark}

\numberwithin{equation}{section}

\newcommand{\Q}{\mathbb{Q}}

\newcommand{\Z}{\mathbb{Z}}
\newcommand{\calo}{\mathcal{O}}
\newcommand{\ol}{\overline}

\newcommand{\Gal}{\mathrm{Gal}}
\newcommand{\Hom}{{\mathrm{Hom}}}
\DeclareMathOperator{\GL}{GL}
\DeclareMathOperator{\id}{id}
\DeclareMathOperator{\lcm}{lcm}
\DeclareMathOperator{\ord}{ord}
\DeclareMathOperator{\cok}{cok}
\DeclareMathOperator{\Cl}{Cl}
\DeclareMathOperator{\M}{M}

\newcommand{\starpar}[1]{%
  \par\vspace{\abovedisplayskip}%
  \noindent
  \makebox[1cm][l]{$(\star)$}%
  \parbox[t]{\dimexpr\linewidth-2cm}{#1}%
  \par\vspace{\belowdisplayskip}%
}

\makeatletter
\AtBeginDocument
 {
    \def\@thm#1#2#3{%
      \ifhmode
        \unskip\unskip\par
      \fi
      \normalfont
      \trivlist
      \let\thmheadnl\relax
      \let\thm@swap\@gobble
      \let\thm@indent\indent % indent
      \thm@headfont{\scshape}% heading font small caps
      \thm@notefont{\fontseries\mddefault\upshape}%
      \thm@headpunct{.}% add period after heading
      \thm@headsep 5\p@ plus\p@ minus\p@\relax
      \thm@space@setup
      #1% style overrides
      \@topsep \thm@preskip               % used by thm head
      \@topsepadd \thm@postskip           % used by \@endparenv
      \def\dth@counter{#2}%
      \ifx\@empty\dth@counter
        \def\@tempa{%
          \@oparg{\@begintheorem{#3}{}}[]%
        }%
      \else
        \H@refstepcounter{#2}%
        \hyper@makecurrent{#2}%
        \let\Hy@dth@currentHref\@currentHref
        \AddToHookNext{para/begin}{\MakeLinkTarget*{\Hy@dth@currentHref}}%
        \def\@tempa{%
          \@oparg{\@begintheorem{#3}{\csname the#2\endcsname}}[]%
        }%
      \fi
      \@tempa
    }%
  \dth@everypar={%
    \@minipagefalse \global\@newlistfalse
    \@noparitemfalse
    \if@inlabel
      \global\@inlabelfalse
      \begingroup \setbox\z@\lastbox
       \ifvoid\z@ \kern-\itemindent \fi
      \endgroup
      \unhbox\@labels
    \fi
    \if@nobreak \@nobreakfalse \clubpenalty\@M
    \else \clubpenalty\@clubpenalty \everypar{}%
    \fi
  }%
}

\title[The Discrete Logarithm Problem in Cokernels of $\calo_K$-Matrices]{The Discrete Logarithm Problem in Cokernels of $\calo_K$-Matrices}

\author{Isaac Rajagopal}
\address{Massachusetts Institute of Technology}
\email{isaacraj@mit.edu}

\date{\today}

\begin{document}

\begin{abstract}
In 2009 and 2010, Blackburn and Shokrieh independently found that the discrete logarithm can be computed efficiently on the sandpile group of a graph, meaning that sandpile groups are not secure for cryptography. We generalize this problem to cokernels of matrices with entries in the ring of integers $\mathcal{O}_K$ of a number field $K$. When $K$ has nontrivial class group, the failure of the Euclidean algorithm in $\mathcal{O}_K$ is an obstacle to generalizing previous methods. For $M$ in $\M_{n\times m}(\mathcal{O}_K)$, we overcome this obstacle to efficiently compute discrete logarithms in $\mathrm{cok}(M) = \mathcal{O}_K^n/M\mathcal{O}_K^m$. In particular, we find an algorithm with time complexity $\tilde{O}((m+n)^{\omega+1})$, where $\omega$ is an exponent of matrix multiplication, to compute discrete logarithms in $\mathrm{cok}(M)$ when $\mathrm{cok}(M)$ is viewed either as an $\mathcal{O}_K$-module or as a group. When $M$ is Hermitian with respect to a Galois involution $\sigma$ and nonsingular, we improve the time complexity to $\tilde{O}(n^\omega)$.
\end{abstract}

\maketitle

\numberwithin{theorem}{section}

\section{Introduction}

\subsection{Discrete logarithm problem}
Let $(G,+)$ be an abelian group. The discrete logarithm problem (DLP) in $G$ can be stated as follows.
\begin{problem}[Discrete Logarithm Problem (DLP)]\label{dlpgroup}
    Given $g,h$ in $G$, find some $x \in \Z$ such that $xg = h$, or determine that no such $x$ exists.
\end{problem}

\begin{remark}
    The difficulty of the DLP depends on the way that group elements are represented as data, so there are often isomorphic groups $G$ and $G'$ such that the DLP is easy in $G$ but hard in $G'$. For example, the DLP is easy to solve in $\Z/(p-1)\Z$ using the Euclidean algorithm but is harder to solve in $(\Z/p\Z)^\times$ (see \cite[Section 8.3]{KL15}).
\end{remark}

Many modern cryptographic systems rely on the existence of groups in which it is easy to compute the group operations but difficult to solve the DLP quickly. Note that the security of such groups is always conjectural since we cannot prove that solving the DLP is hard. Number theory is rife with examples of such groups, such as $(\Z/p\Z)^\times$, elliptic curve groups, class groups of quadratic number fields, and Jacobians of algebraic curves. We refer interested readers to \cite{KL15} or \cite{Buchmann00} or \cite{CFGRD06}.

\subsection{Discrete logarithm is easy for sandpile groups}
For a generic group $G$, the DLP can be solved in $O(|G|^{\frac{1}{2}})$ group operations, using algorithms such as the baby-step giant-step algorithm (see \cite[Chapter 19]{CFGRD06}). Elliptic curves are widely used in cryptography because there is no known algorithm for computing the DLP quickly on them. Any graph $\Gamma$ has a naturally associated group $S_{\Gamma}$ called the sandpile group, the Jacobian, or the Picard group of the graph. Because of analogies between the sandpile group of a graph and groups associated to algebraic curves, Biggs \cite{Biggs07} asked whether the sandpile group of a graph would be a good candidate for public-key cryptography. Shokrieh \cite{Shokrieh2010} and Blackburn \cite{Blackburn09} answered Biggs' question in the negative by finding efficient algorithms to compute the DLP on these groups. The size of $S_{\Gamma}$, which is equal to the number of spanning trees of $\Gamma$, can be exponential in the number of vertices of $\Gamma$ (see \ \cite{Biggs07}). Shokrieh and Blackburn's methods, which are polynomial in the number of vertices of $\Gamma$, are therefore much faster than the generic algorithms.

If $Q$ is a reduced Laplacian matrix of a connected graph $\Gamma$ with $n+1$ vertices, then the sandpile group can be defined by $S_{\Gamma} = \Z^{n}/Q\Z^{n}$.\footnote{Label the vertices of $\Gamma$ as $v_1,\ldots,v_{n+1}$. The $(n+1) \times (n+1)$ Laplacian matrix $L$ is formed by $L_{ij}$ being $-1$ times the number of edges connecting vertices $i$ and $j$ if $i \neq j$, and $L_{ii} = \deg v_i$. To form $Q$ from $L$, choose $1 \leq k \leq n+1$ and delete the $k^{th}$ row and column. The resulting group $S_{\Gamma}$ does not depend on the choice of $k$.} We refer the interested reader to \cite{Klivans19} or \cite{CP18} for a full overview of sandpile groups. Shokrieh's \cite[Algorithm~4.1]{Shokrieh2010} solution to the DLP on cyclic sandpile groups can be computed in $O(n^{\omega})$ integral operations, where $\omega$ is an exponent of matrix multiplication. In \cite[Remark 5.4]{Shokrieh2010}, Shokrieh used the extended Euclidean algorithm to generalize this method to noncyclic sandpile groups. The main tool used in Shokrieh's proof is the canonical perfect symmetric bilinear pairing $\langle \cdot, \cdot \rangle: S_{\Gamma} \times S_{\Gamma} \to \Q/\Z$, computable in $O(n^{\omega})$ integral operations \cite[Theorem 3.4, Proposition 3.7]{Shokrieh2010}. Independently, Blackburn \cite{Blackburn09} also found an argument to compute the DLP on $S_{\Gamma}$ in polynomial time in $n$ for specific examples of $\Gamma$, based on putting $Q$ into Smith normal form. 

In this paper, we generalize both Blackburn and Shokrieh's methods to solve the DLP for cokernels of matrices over the ring of integers $\calo_K$ of a number field $K$. The failure of the Euclidean algorithm over $K$ when $\Cl(K)$ is nontrivial suggests that Blackburn and Shokrieh's techniques may not apply in this setting. We circumvent this by using the structure of $\calo_K$ as a free $\Z$-module. 

\subsection{Cokernels of Rectangular Matrices}
For a ring\footnote{All rings are assumed to be commutative.} $R$, we can generalize the DLP to $R$-modules.

\begin{problem}[DLP on $R$-modules]\label{dlpRmod}
    Let $G$ be an $R$-module with $g,h$ in $G$. Then find some $\chi \in R$ such that $\chi g = h$, or determine that no such $\chi$ exists.
\end{problem}

Problem \ref{dlpRmod} is similar to Problem \ref{dlpgroup}, except that the scalar $\chi$ can be any element of $R$ rather than only an integer. In this article we will always take $R = \calo_K$ to be the ring of integers of a number field, but it would be interesting to study Problem \ref{dlpRmod} over other rings $R$.

\begin{remark}\label{rmkRmodules}
  We generalize the Diffie--Hellman \cite{DH76, Merkle78, ElGamal85} key exchange to work over an arbitrary ring $R$; this can be cracked if Problem \ref{dlpRmod} is solved. Let $g \in G$ be shared between Alice and Bob, and let Alice choose $\alpha$ in $R$ and Bob choose $\beta$ in $R$. Then suppose Alice shares $\alpha g$ with Bob, and Bob shares $\beta g$ with Alice, across public channels. Alice and Bob can both calculate $\alpha \beta g = \beta \alpha g$. However, someone watching the channel only knows $g$, $\alpha g$, and $\beta g$, which seems not to be enough information to calculate $\alpha$ or $\beta$ or $\alpha \beta g$ unless Problem \ref{dlpRmod} is solved.\footnote{See \cite[Section 8.3.2]{KL15} for a discussion comparing the difficulty of the DLP to the difficulty of cracking the Diffie--Hellman key exchange.}
\end{remark}

We study cokernels of matrices $M \in M_{n \times m}(\calo)$, where $\calo = \calo_K$ is the ring of integers of a fixed number field $K$ of degree $d$. Viewing $M$ as a linear map $\calo^m \to \calo^n$, we may regard $\cok(M)$ as $\calo^n/M\calo^m$, which carries the structure of both an $\calo$-module and an abelian group. Any finitely generated $\calo$-module is isomorphic to $\cok(M)$ for some such matrix $M$. Since the class group $\Cl(K)$ of $K$ can be realized as a finite cokernel, the group $\cok(M)$ is a good model for $\Cl(K)$ (see \cite{Wood15}). Using the structure of $\calo$ as a free $\Z$-module, we solve Problem \ref{dlpRmod} (and hence Problem \ref{dlpgroup}) in $\cok(M)$ in polynomial time in $m+n$. 

To state this theorem, let $\omega$ be the smallest value such that two $n \times n$ matrices can be multiplied in $O(n^\omega)$ operations; the best known bound for $\omega$ is $\omega < 2.371339$ \cite{MaMu25}. Define $f(n) \in \tilde{O}(n^\theta)$ to mean that $f(n) \in O(n^\theta\log^k(n))$ for some $k \geq 0$. Our asymptotic $\tilde{O}$ notation will hide factors depending on $d = [K:\Q]$.

\begin{theorem}\label{thmblackburn}
    Let $M \in M_{n \times m}(\calo)$, and let $g,h \in \cok(M)$. Then, in $\tilde{O}((m+n)^{\omega + 1})$ operations in $\Z$, we can explicitly describe all $\chi \in \calo$ such that $\chi g = h$.
\end{theorem}

\subsection{Cokernels of Hermitian Matrices}
Because sandpile groups of connected graphs are cokernels of nonsingular integral symmetric matrices, a natural generalization to larger number fields $K$ is given by cokernels of nonsingular Hermitian matrices $M$ in $\M_{n}(\calo)$ \cite{Wood23,Lee23,Yan23,Hodges25}. To define Hermitian matrices, fix $\sigma \in \Gal(K/\Q)$ with $\sigma^2= \id$. Then, $M$ is Hermitian (with respect to $\sigma$) if $\sigma(M^t) = M$, where $M^t$ denotes the transpose of $M$. In forthcoming work, Hodges \cite{Hodges25} finds a canonical perfect Hermitian pairing on $\cok(M)$. We use this pairing to extend Shokrieh's methods to solve Problem \ref{dlpgroup} in $\cok(M)$, with fewer operations in $\Z$ than Theorem \ref{thmblackburn}.

\begin{theorem}\label{thmshokrieh}
Let $M \in \M_n(\calo)$ be a nonsingular Hermitian matrix, and let $g$ and $h$ be elements of $\cok(M)$. Then, in $\tilde O(n^{\omega})$ operations in $\Z$, we can:
\begin{enumerate}[(a)]
    \item determine whether there exists $x \in \Z$ such that $x g = h$;
    \item if such an $x$ exists, find $x_0 \in \Z$ and $\ord(g) \in \Z$ such that $x g = h$ if and only if $x \in x_0+(\ord(g))$.
\end{enumerate}
\end{theorem}

Letting $K = \Q$, $\sigma = \id$, and $M = Q$ be the reduced Laplacian of a graph, Theorem \ref{thmshokrieh} returns the sandpile case of \cite{Shokrieh2010} with a time complexity of $\tilde O(n^{\omega})$ replacing the $O(n^{\omega})$ in \cite{Shokrieh2010}.

\subsection{Methods and Outline}
In Section \ref{sectblackburn}, we prove Theorem \ref{thmblackburn} about cokernels of rectangular matrices over $\calo$. We use the structure of $\calo \simeq \Z^d$ to reduce the equation $\chi g = h$ to a system of linear equations over $\Z$. We then use results from \cite{Storjohann2000} to solve this system of linear equations by using Hermite normal form, a more efficiently computable but weaker form of Smith normal form. This can be viewed as a generalization of the methods in \cite{Blackburn09}, which involve converting $M$ (for very specific matrices $M$) to Smith normal form.

In Section \ref{sectshokrieh}, we prove Theorem \ref{thmshokrieh} about cokernels of Hermitian matrices over $\calo$ by generalizing Shokrieh's \cite{Shokrieh2010} methods over $\Z$. We first show that the perfect pairing ${\langle \cdot, \cdot \rangle:~\cok(M) \times \cok(M)~\to~K/\calo}$ found by Hodges \cite{Hodges25} is computable in $\tilde O(n^{\omega})$ operations in $\Z$, using the theory of generalized inverses of matrices. We use this to reduce solving the DLP to a system of linear equations involving the pairings $\langle g, v_i \rangle$, where $v_1,\ldots,v_n$ are images of the basis vectors of $\calo^n$ in $\cok(M)$. We then reduce these equations to a system of linear equations in $\Z$, which we can easily solve using the extended Euclidean algorithm. 

\subsection{The role of AI in this paper} In an earlier draft, we used methods similar to Section \ref{sectshokrieh} to solve Problem \ref{dlpRmod} in the torsion submodule of $\cok(M)$, where $M$ is a (possibly singular) Hermitian matrix in $M_n(\calo)$, in $\tilde{O}(n^{\omega+1})$ operations. When prompted with that earlier draft, ChatGPT 5.4 Pro generalized this result to rectangular matrices and simplified its proof, which has become Theorem \ref{thmblackburn}. So, the main proof idea in Section \ref{sectblackburn} comes from ChatGPT. We have independently verified all results in this paper. 

We briefly summarize our methods from that earlier draft here, and we are willing to share that draft upon request. Using the isomorphism $\calo \simeq \Z^d$, we converted $M$ to $M' \in \M_{dn}(\Z)$ and used Hermite normal form (with row operations) to find generators of the torsion submodule of $\cok(M')$. These allowed us to find generators of the torsion submodule of $\cok(M)$, and repeat the conventions in $(\star)$ with these generators replacing $v_1,\ldots,v_n$. We then used similar methods to those in Section \ref{sectshokrieh} to reduce the problem to the system of equations in (\ref{eqgeneral}), with $\sigma(\chi)$ replacing $x$. As these equations are more complicated to solve for $\chi \in \calo$ than $x \in \Z$, the extended Euclidean algorithm was not sufficient. So, we solved these equations in a similar manner to Section \ref{sectblackburn}, by reducing them to a system of equations over $\Z$ and then solving those using Hermite normal form.

\section{Cokernels of Rectangular Matrices}\label{sectblackburn}

We now prove Theorem \ref{thmblackburn} using a simple reduction to a linear algebra problem over $\Z$. The first paragraph of the proof was paraphrased from ChatGPT.

\begin{proof}[Proof of Theorem \ref{thmblackburn}]
    Let $M \in M_{n \times m}(\calo)$. Let $g,h \in \cok(M)$ be represented by arbitrary lifts $G$ and $H$ in $\calo^n$. Then $\chi g = h$ is equivalent to the existence of $Y \in \calo^m$ such that \begin{equation}\label{eqGPT}
        \chi G - H = MY.
    \end{equation} We now see how to reduce (\ref{eqGPT}) to a system of $dn$ linear equations in $\Z$ for $dm+d$ unknowns. 
    
    As $\calo \simeq \Z^d$ as a group, let $e_1,\ldots,e_d$ be a basis for $\calo$ as a $\Z$-module. This gives an isomorphism $\calo^n \to \Z^{dn}$, under which $G$ and $H$ become elements of $\Z^{dn}$. Using the isomorphism $\calo^m \to \Z^{dm}$ as well, $M$ becomes a linear map from $\Z^{dm} \to \Z^{dn}$, which can be represented as a matrix in $M_{dn \times dm}(\Z)$. Expanding out $\chi G$ in terms of the basis $e_1,\ldots,e_d$, (\ref{eqGPT}) becomes a system of $dn$ linear equations in $\Z$ for $dm+d$ unknowns. Here, the unknowns are $Y \in \Z^{dm}$ and $\chi \in \Z^d$. This reduces \eqref{eqGPT} to a matrix equation over $\Z$ given as $A\mathbf{x} = \mathbf{y}$, where $A$ is a $dn \times (dm+d)$ matrix of unknown rank and $\mathbf{y}$ is a vector with $dn$ entries. It is known \cite{Schrijver86,Storjohann2000} how to solve this equation for $\mathbf{x} \in \Z^{dm+d}$ in $\tilde{O}((n+m)^{\omega+1})$ operations in $\Z$. For completeness we summarize the argument here.

    First, we compute the rank $r$ of $A$ and the (row) rank profile of $A$, which is a list of $r$ rows which generate the row space of $A$. This can be done in ${\tilde O((dn+dm+d)^{\omega+1})=\tilde O((n+m)^{\omega+1 })}$ operations using the algorithm in \cite[Chapter 2]{Storjohann2000}. We reorder the rows of $A$ for notational convenience so that these are the first $r$ rows of $A$. (We have to put the rows back in their original order at the end.) We then compute the (column) Hermite normal form of $A$ in $\tilde O((dn+dm+d)^{\omega+1})=\tilde O((n+m)^{\omega+1 })$ operations using the algorithm in \cite[Chapter 6]{Storjohann2000}. This means that we can find $H$ in $\M_{dn \times (dm+d)}(\Z)$ and $U \in \text{GL}_{dm+d}(\Z)$ such that $H = AU$, and $H$ is of the following form: 
\[
H = \begin{pmatrix}
H_{11} & 0      & 0      & \cdots & 0      & 0      & \cdots & 0 \\
H_{21} & H_{22} & 0      & \cdots & 0      & 0      & \cdots & 0 \\
H_{31} & H_{32} & H_{33} & \cdots & 0      & 0      & \cdots & 0 \\
\vdots & \vdots & \vdots & \ddots & \vdots & \vdots & \ddots & \vdots \\
H_{r1} & H_{r2} & H_{r3} & \cdots & H_{rr} & 0      & \cdots & 0 \\
\vdots & \vdots & \vdots & \ddots & \vdots & \vdots & \ddots & \vdots \\
H_{(dn)1} & H_{(dn)2} & H_{(dn)3} & \cdots & H_{(dn)r} & 0 & \cdots & 0
\end{pmatrix}
\]
    with $H_{ii}\neq 0$ for $1 \leq i \leq r$.\footnote{In general, Hermite normal form only requires all columns of zeros to be at the right and the leading coefficient of each nonzero column to be below the the leading coefficient of the column to its left. However, after our reordering of the rows of $A$, the topmost $r$ rows will generate the row space, which forces the given characterization.} The matrix $U$ should be thought of as a set of column operations to perform on $A$.
    
    We now follow \cite[Corollary 5.3b]{Schrijver86} to describe how to solve $A\mathbf{x} = \mathbf{y}$ from its Hermite normal form $H$. Solving $A\mathbf{x} = \mathbf{y}$ is equivalent to solving $H\mathbf{z} = \mathbf{y}$ for $\mathbf{z} \in \Z^{dm+d}$, and then taking $\mathbf{x} = U \mathbf{z}$. As only the leftmost $r$ columns of $H$ are nonzero, we only need to consider the uppermost $r$ entries of $\mathbf{z}$, as the others can be any integers. Call this truncation $\ol{\mathbf{z}} \in \Z^r$. Let $N$ be the $r \times r$ minor of $H$ formed from the first $r$ columns and the first $r$ rows. Then, $N$ is lower triangular with nonzero diagonal entries. Let $\ol{\mathbf{y}} \in \Z^r$ be formed from the entries of $\mathbf{y}$ in the first $r$ rows. 
    
    Then, any solution to $H\mathbf{z} = \mathbf{y}$ will also yield a solution to $N \ol{\mathbf{z}} = \ol{\mathbf{y}}$. Since $N$ is lower triangular with nonzero diagonal entries, by back substitution there is a unique $\ol{\mathbf{z_0}} \in \Q^r$ such that $N \ol{\mathbf{z_0}} = \ol{\mathbf{y}}$. If $\ol{\mathbf{z_0}}$ is not in $\Z^r$, there are no solutions to $H\mathbf{z} = \mathbf{y}$ with $\mathbf{z} \in \Z^{dm+d}$.  If $\ol{\mathbf{z_0}}$ is in $\Z^r$, let $\mathbf{z_0} \in \Z^{dm+d}$ be formed by appending zeros to $\ol{\mathbf{z_0}}$. If $H\mathbf{z_0} \neq \mathbf{y}$, then there are also no solutions to to $H\mathbf{z} = \mathbf{y}$ with $\mathbf{z} \in \Z^{dm+d}$. If $\ol{\mathbf{z_0}}$ is in $\Z^r$ and $H\mathbf{z_0} = \mathbf{y}$, then the solutions to $A\mathbf{x} = \mathbf{y}$ are given by all $\mathbf{x}$ of the form $\mathbf{x}=U\mathbf{z}$, where the first $r$ coordinates of $\mathbf{z}$ are equal to $\ol{\mathbf{z_0}}$.
\end{proof}

\section{Cokernels of Hermitian Matrices}\label{sectshokrieh}

Throughout this section, let $M \in \M_n(\calo)$ be nonsingular and Hermitian with respect to $\sigma$.

\subsection{Computing the Pairings}

We now define the generalized inverse of a matrix, which will be necessary for defining the pairing on $\cok(M)$.
\begin{definition}\cite[Definition 1.1.1]{BG03}
    For an arbitrary matrix $A$ with entries in $K$, define a \emph{generalized inverse} of $A$ to be any matrix $L$ with entries in $K$ satisfying $ALA = A.$
\end{definition}
We begin by establishing the time complexity of computing a generalized inverse.
\begin{lemma}\label{propgeninverse}
    For $A \in \M_n(K)$, a generalized inverse of $A$ can be computed in $\tilde O(n^{\omega})$ operations in $\Z$.
\end{lemma}
\begin{proof}
    Let $r$ be the rank of $A$, and let $I_r$ be an $r \times r$ identity matrix. By \cite[Proposition 16.13]{BCS97}, we can compute matrices $S,T \in \GL_n(K)$ such that \[SAT = \begin{bmatrix}
        I_r & 0 \\
        0 & 0
    \end{bmatrix},\] using $\tilde O(n^{\omega})$ operations in $K$. Then, $A = S^{-1}\begin{bmatrix}
        I_r & 0 \\
        0 & 0
    \end{bmatrix}T^{-1}$. It is easy to check that $L = T \begin{bmatrix}
        I_r & 0 \\
        0 & 0
    \end{bmatrix} S$ is a generalized inverse of $A$ (see {\ }\cite[Theorem 1.2.1]{BG03}). Since $K$ is a degree $d$ vector space over $\Q$, operations in $K$ can be computed with a constant number of operations in $\Z$.
\end{proof}

We now define Hermitian pairings and what it means for them to be perfect. 

\begin{definition}
For an $\mathcal{O}$-module $G$, a \emph{Hermitian pairing} (with respect to $\sigma$ in $\Gal(K/\Q)[2]$) is a map $\langle \cdot , \cdot \rangle :~G \times G \to K / \mathcal{O}$ satisfying the following:
\begin{enumerate}
    \item $\langle g_1 + g_2, h_1+ h_2 \rangle = \langle g_1, h_1 \rangle + \langle g_1, h_2 \rangle + \langle g_2, h_1 \rangle + \langle g_2, h_2 \rangle$ for all $g_1, g_2, h_1, h_2 \in G$;
    \item $\langle rg, sh \rangle = \sigma(r)s \langle g, h \rangle$ for all $r,s \in \mathcal{O}$ and $g,h \in G$;
    \item $\sigma(\langle h, g \rangle) = \langle g, h \rangle$ for all $g,h \in G$.
\end{enumerate}
This pairing is said to be \emph{perfect} if the induced homomorphism of abelian groups $G \to \Hom_\calo(G,K/\calo)$ defined by $g \mapsto \langle g,\cdot\rangle$ is an isomorphism.
\end{definition}

Hodges \cite{Hodges25} defines a pairing on $\cok(M)$ by generalizing the pairing in \cite[Section 1]{BL02} over $\Q$ to work over $K$. Furthermore, this pairing is perfect.

\begin{lemma}\label{proppairing}\cite{Hodges25} 
    Let $M \in \M_n(\calo)$ be nonsingular and Hermitian with respect to $\sigma$. Let $\tau$, $\tau'$ be in $\cok(M)$, and choose lifts of $\tau$ and $\tau'$ to elements $T, T'$ in $\calo^n$, respectively. Then, there exist nonzero $k,k'$ in $\calo$ and $S, S' \in \calo^n$ such that $kT = MS$ and $k'T' = MS'$. Define the pairing $\langle \cdot , \cdot \rangle:\cok(M) \times \cok(M) \to K/\calo$ by \[\langle \tau,\tau'\rangle \coloneq \frac{\sigma(S^t)MS'}{\sigma(k)k'}.\] Then, $\langle \cdot , \cdot \rangle$ is a well-defined perfect Hermitian pairing on $\cok(M)$. 
\end{lemma}

\begin{remark}
If $M$ is invertible, then $\langle \tau,\tau'\rangle=\sigma(T^t)M^{-1}T'$. We will see this in (\ref{eq101}).
\end{remark}

We can generalize the work of \cite[Proposition 3.7]{Shokrieh2010} to efficiently compute the pairings we will need from the generalized inverse of $M$. 

\begin{lemma}\label{propcomputing}
Let  $f \in \cok(M)$, where $M \in \M_n(\calo)$ is Hermitian and nonsingular. Let $v_1,\ldots,v_n$ be elements of $\cok(M)$. Then the pairings $\langle f , v_1 \rangle,\ldots,\langle f,v_n\rangle$ as defined in Lemma~\ref{proppairing} can be computed with $\tilde O(n^{\omega})$ operations in $\Z$.
\end{lemma}

\begin{proof}
Choose $F$ and $V_i$ in $\calo^n$ to be lifts of $f$ and $v_i$, respectively. Let $L$ be a generalized inverse of $M$. Following the notation of Lemma~\ref{proppairing} with $\tau = f$ and $\tau' = v_i$ and $T = F$ and $T' = V_i$, 
\begin{equation}\label{eq101}
   \langle f,v_i\rangle = \frac{\sigma(S^t)MS'}{\sigma(k)k'} \\
    = \frac{\sigma(S^t)M L MS'}
    {\sigma(k)k'}\\
    = \sigma\left(\frac{(MS)^t}{k}\right) L \left(\frac{MS'}{k'} \right) \\
    = \sigma(F^t)LV_i. 
\end{equation}
By Lemma \ref{propgeninverse}, $L$ can be computed in $\tilde O(n^{\omega})$ operations in $\Z$. So the $1 \times n$ row vector $\sigma(F^t)L$ with entries in $K$ can be computed with $O(n^{2})$ further operations in $K$, and hence $O(n^2)$ operations in $\Z$. Then, for each $i \in \{1,\ldots,n\}$, computing $\langle f,v_i\rangle$ will take $O(n)$ further operations in $K$, and hence $O(n)$ operations in $\Z$. So, it takes $O(n^2)$ more operations in $\Z$ to compute $\langle f , v_1 \rangle,\ldots,\langle f,v_n\rangle$.
\end{proof}

\subsection{Solving the DLP}
We now adopt the following conventions:
\starpar{Let $M \in \M_n(\calo)$ be Hermitian with respect to $\sigma$. Let $g,h \in \cok(M)$ and let $v_1,\ldots,v_n$ be the images of the canonical basis vectors of $\calo^n$ inside $\cok(M)$. For $1 \leq i \leq n$, fix $\alpha_i,\gamma_i \in \calo$ and $b_i \in \Z$ such that $\langle g,v_i\rangle \equiv \frac{\alpha_i}{b_i} \pmod{\calo}$ and $\langle h,v_i\rangle \equiv \frac{\gamma_i}{b_i} \pmod{\calo}$.}

Using Lemma \ref{propcomputing} with both $f = g$ and $f=h$ allows us to compute $\alpha_i$ and $b_i$ and $\gamma_i$ for all $1 \leq i \leq n$ in $\tilde O(n^{\omega})$ operations in $\Z$. 

Let $\ord(g)$ be the minimum $i > 0$ such that $ig = 0$. If we can find $x_0 \in \Z$ such that $x_0 g = h$, then \[\{x \in \Z : x g = h\} = x_0 + (\ord(g)).\] We now find $\ord(g)$ in terms of $\alpha_i$ and $b_i$, in particular as the least common multiple of the denominators of the fractions $\frac{\alpha_i}{b_i}$, written in lowest terms. 

\begin{lemma}\label{lemma1}
    Assume the conventions of $(\star)$. Let $\alpha_i' \in \calo$ and $b_i' \in \Z_{>0}$ such that $\frac{\alpha_i}{b_i} =  \frac{\alpha_i'}{b_i'}$ and there is no prime $p$ such that $p \mid b_i'$ and $(\alpha_i') \subseteq (p)$. Then $\ord(g) = \lcm(b_1',\ldots,b_n').$
\end{lemma}
\begin{proof}
    Let $e \in \cok(M)$. By Lemma \ref{proppairing}, the pairing $\langle \cdot,\cdot\rangle$ is perfect, so $e = 0$ if and only if $\langle e,f\rangle = 0$ for all $f \in \cok(M)$. As $v_1,\ldots,v_n$ generate $\cok(M)$, then $e=0$ if and only if $\langle e, v_i \rangle \equiv 0 \pmod{\calo}$ for all $1 \leq i \leq n$. 

    Using that $\langle \cdot,\cdot\rangle$ is Hermitian and $\sigma|_{\Z} = \id|_{\Z}$, \begin{align*}(\ord(g)) &\coloneq \{x \in \Z: xg = 0\} \\ 
    &= \{x \in \Z:\langle x g,v_i\rangle \equiv 0\text{ mod}(\calo) \text{ for all } 1\leq i \leq n\} \\
    &= \{x \in \Z:\frac{\sigma(x) \alpha_i'}{b_i'} \in \calo \text{ for all } 1\leq i \leq n\} \\
    &= \{x \in \Z: x\alpha_i' \in (b_i') \text{ for all } 1\leq i \leq n\}.
    \end{align*}

    We now show that, for $x \in \Z$, $x \alpha_i' \in (b_i')$ if and only if $b_i' \mid x$.

     Let $e_1,\ldots,e_d$ be our basis for $\calo$ as $\Z$-module. Then, let $\alpha_i' = a_1e_1+\cdots +a_de_d$ for $a_j \in \Z$. Therefore, \[x \alpha_i'= x a_1e_1+\cdots +x a_de_d.\] Notice that $(b_i')$ is the set of all elements of the form $r_1e_1+\cdots+r_de_d$, where $r_j \in \Z$ with $b_i' \mid r_j$ for all $j$. Because $x \alpha_i' \in (b_i')$, then $b_i' \mid xa_j$ for all $1\leq j\leq d$. By assumption, there is no prime $p$ which divides $b_i'$ and $a_j$ for all $j$. So, for all prime powers $p^k$ dividing $b_i'$, we can find some $j$ such that $p \nmid a_j$ and $p^k \mid x a_j$. Hence, $p^k \mid x$. So, $b_i' \mid x$.

    Using the work in the last paragraph, we have
    \[ (\ord(g)) = \bigcap_{i=1}^n \{x \in \Z:x \alpha_i' \in (b_i')\}= \bigcap_{i=1}^s (b_i')= (\lcm(b_1',\ldots,b_n')),\] so we are done.
\end{proof}

We are now ready to prove Theorem \ref{thmshokrieh}, using the definitions in $(\star)$ to reduce the problem to a simple system of equations which can be solved using the extended Euclidean algorithm.

\begin{proof}[Proof of Theorem \ref{thmshokrieh}]
    Clearly, $x g = h$ if and only if $h - x g =0$. By the same logic in Lemma \ref{lemma1}, as $\langle \cdot,\cdot\rangle$ is perfect, then $h - x g =0$ if and only if $\langle h -x g,v_i \rangle \equiv 0 \pmod{\calo}$ for all $1 \leq i \leq n$. Because $\langle \cdot,\cdot\rangle$ is Hermitian and $\gamma_i/b_i$ is an arbitrary lift of $\langle h,v_i\rangle$ to $K$, 
    \[\langle h -x g,v_i \rangle \equiv \langle h,v_i\rangle-\frac{\sigma(x)\alpha_i}{b_i} \equiv \frac{\gamma_i}{b_i}-\frac{x\alpha_i}{b_i} \pmod{\calo}.\]
    So, $x g = h$ if and only if $\frac{\gamma_i}{b_i}-\frac{x\alpha_i}{b_i} \in \calo$ for all $1 \leq i \leq n$. This is equivalent to there existing $\psi_1,\ldots,\psi_n$ in $\calo$ that solve the system of equations \begin{equation}\label{eqgeneral}
        x \alpha_i + b_i\psi_i = \gamma_i \text{ for }1\leq i\leq n.
    \end{equation}
    Letting $e_1,\ldots,e_d$ be a basis for $\calo$ as a $\Z$-module, let
    \begin{equation*}
      \begin{aligned}
        \alpha_i &= a_i^{(1)}e_1 + \cdots + a_i^{(d)}e_d  \\
        \psi_i &= y_i^{(1)}e_1 + \cdots + y_i^{(d)}e_d \\
        \gamma_i &= c_i^{(1)}e_1 + \cdots + c_i^{(d)}e_d.
    \end{aligned}
    \label{eq:defs}
  \end{equation*}
  Then (\ref{eqgeneral}) is equivalent to the system of $dn$ equations over $\Z$ given by 
  \begin{equation}\label{eqbig2}
xa_i^{(k)}+b_iy_i^{(k)}=c_i^{(k)} \text{ for } 1\leq i \leq n,1\leq k \leq d\;.
    \end{equation}
The equations in \eqref{eqbig2} can be viewed as $dn$ congruences of the form $x a_i^{(k)} \equiv c_i^{(k)}\pmod{b_i}$. So, we can determine whether there exists $x \in \Z$ satisfying (\ref{eqbig2}) and find such an $x$ if it exists using the extended Euclidean algorithm and the Chinese Remainder Theorem, which can be run in $O(dn) = O(n)$ operations in $\Z$. Combining this with Lemma \ref{lemma1} yields the theorem.
\end{proof}

\section*{Acknowledgements}

This research was conducted at the 2025 University of Minnesota Duluth REU with support from Jane Street Capital, NSF Grant 2409861, and donations from Ray Sidney and Eric Wepsic. I thank Joe Gallian and Colin Defant for providing this wonderful opportunity. I thank Eliot Hodges for suggesting this project and advising my whole research process, during which he gave detailed feedback and suggestions. I thank Tommy Hofmann, Mitchell Lee, Nathan Sheffield, and Arne Storjohann for many helpful suggestions. Finally, I thank Eliot Hodges, Noah Kravitz, Mitchell Lee, Rupert Li, and Maya Sankar for advising the whole Duluth REU.

%%%%%%%%%%%%%%%%%%%%%%%%%%%% References %%%%%%%%%%%%%%%%%%%%%%%%%%%%

\bibliography{bibliography}{}

@article {Shokrieh2010,
    AUTHOR = {Shokrieh, Farbod},
     TITLE = {The monodromy pairing and discrete logarithm on the {J}acobian
              of finite graphs},
   JOURNAL = {J. Math. Cryptol.},
  FJOURNAL = {Journal of Mathematical Cryptology},
    VOLUME = {4},
      YEAR = {2010},
    NUMBER = {1},
     PAGES = {43--56},
      ISSN = {1862-2976,1862-2984},
   MRCLASS = {05C50 (14H40 94A60)},
  MRNUMBER = {2660333},
       DOI = {10.1515/JMC.2010.002},
       URL = {https://doi.org/10.1515/JMC.2010.002},
}

@article{Merkle78,
author = {Merkle, Ralph C.},
title = {Secure communications over insecure channels},
year = {1978},
issue_date = {April 1978},
publisher = {Association for Computing Machinery},
address = {New York, NY, USA},
volume = {21},
number = {4},
issn = {0001-0782},
url = {https://doi.org/10.1145/359460.359473},
doi = {10.1145/359460.359473},
journal = {Commun. ACM},
month = apr,
pages = {294–299},
numpages = {6}
}

@ARTICLE{DH76,
  author={Diffie, Whitfield and Hellman, Martin},
  journal={IEEE Trans. Inform. Theory}, 
  title={New directions in cryptography}, 
  year={1976},
  volume={22},
  number={6},
  pages={644-654},
  doi={10.1109/TIT.1976.1055638}}

@article {Biggs07,
    AUTHOR = {Biggs, Norman},
     TITLE = {The critical group from a cryptographic perspective},
   JOURNAL = {Bull. Lond. Math. Soc.},
  FJOURNAL = {Bulletin of the London Mathematical Society},
    VOLUME = {39},
      YEAR = {2007},
    NUMBER = {5},
     PAGES = {829--836},
      ISSN = {0024-6093,1469-2120},
   MRCLASS = {94A60 (05C25 20K01)},
  MRNUMBER = {2365232},
MRREVIEWER = {Marc\ M.\ Gysin},
       DOI = {10.1112/blms/bdm070},
       URL = {https://doi.org/10.1112/blms/bdm070},
}

@article {ElGamal85,
    AUTHOR = {El{G}amal, Taher},
     TITLE = {A public key cryptosystem and a signature scheme based on
              discrete logarithms},
   JOURNAL = {IEEE Trans. Inform. Theory},
  FJOURNAL = {Institute of Electrical and Electronics Engineers.
              Transactions on Information Theory},
    VOLUME = {31},
      YEAR = {1985},
    NUMBER = {4},
     PAGES = {469--472},
      ISSN = {0018-9448,1557-9654},
   MRCLASS = {94A60 (11T71)},
  MRNUMBER = {798552},
       DOI = {10.1109/TIT.1985.1057074},
       URL = {https://doi.org/10.1109/TIT.1985.1057074},
}

@article {Blackburn09,
    AUTHOR = {Blackburn, Simon R.},
     TITLE = {Cryptanalysing the critical group: efficiently solving
              {B}iggs's discrete logarithm problem},
   JOURNAL = {J. Math. Cryptol.},
  FJOURNAL = {Journal of Mathematical Cryptology},
    VOLUME = {3},
      YEAR = {2009},
    NUMBER = {3},
     PAGES = {199--203},
      ISSN = {1862-2976,1862-2984},
   MRCLASS = {94A60 (05C50)},
  MRNUMBER = {2604686},
       DOI = {10.1515/JMC.2009.010},
       URL = {https://doi.org/10.1515/JMC.2009.010},
}

@book {KL15,
    AUTHOR = {Katz, Jonathan and Lindell, Yehuda},
     TITLE = {Introduction to modern cryptography},
    SERIES = {Chapman \& Hall/CRC Cryptography and Network Security},
   EDITION = {Second},
 PUBLISHER = {CRC Press, Boca Raton, FL},
      YEAR = {2015},
     PAGES = {xx+583},
      ISBN = {978-1-4665-7026-9},
   MRCLASS = {94-02 (11T71 94A60 94A62)},
  MRNUMBER = {3287369},
MRREVIEWER = {Maura\ Beth\ Paterson},
}

@book {Schrijver86,
    AUTHOR = {Schrijver, Alexander},
     TITLE = {Theory of linear and integer programming},
    SERIES = {Wiley-Interscience Series in Discrete Mathematics},
      NOTE = {A Wiley-Interscience Publication},
 PUBLISHER = {John Wiley \& Sons, Ltd., Chichester},
      YEAR = {1986},
     PAGES = {xii+471},
      ISBN = {0-471-90854-1},
   MRCLASS = {90C05 (90C10)},
  MRNUMBER = {874114},
MRREVIEWER = {J\"urgen\ K\"ohler},
}

@incollection {Buchmann00,
    AUTHOR = {Buchmann, Johannes and Hamdy, Safuat},
     TITLE = {A survey on {IQ} cryptography},
 BOOKTITLE = {Public-key cryptography and computational number theory
              ({W}arsaw, 2000)},
     PAGES = {1--15},
 PUBLISHER = {de Gruyter, Berlin},
      YEAR = {2001},
      ISBN = {3-11-017046-9},
   MRCLASS = {94A60 (94-02)},
  MRNUMBER = {1881623},
}

@book {BG03,
    AUTHOR = {Ben-Israel, Adi and Greville, Thomas N. E.},
     TITLE = {Generalized inverses},
    SERIES = {CMS Books in Mathematics/Ouvrages de Math\'ematiques de la
              SMC},
    VOLUME = {15},
   EDITION = {Second},
      NOTE = {Theory and applications},
 PUBLISHER = {Springer-Verlag, New York},
      YEAR = {2003},
     PAGES = {xvi+420},
      ISBN = {0-387-00293-6},
   MRCLASS = {15A09 (47A05 47A50)},
  MRNUMBER = {1987382},
}

@inproceedings {MaMu25,
    AUTHOR = {Alman, Josh and Duan, Ran and Vassilevska Williams, Virginia
              and Xu, Yinzhan and Xu, Zixuan and Zhou, Renfei},
     TITLE = {More asymmetry yields faster matrix multiplication},
 BOOKTITLE = {Proceedings of the 2025 {A}nnual {ACM}-{SIAM} {S}ymposium on
              {D}iscrete {A}lgorithms ({SODA})},
     PAGES = {2005--2039},
 PUBLISHER = {SIAM, Philadelphia, PA},
      YEAR = {2025},
      ISBN = {978-1-61197-832-2},
   MRCLASS = {68W30},
  MRNUMBER = {4863478},
       DOI = {10.1137/1.9781611978322.63},
       URL = {https://doi-org.libproxy.mit.edu/10.1137/1.9781611978322.63},
}

@book {BCS97,
    AUTHOR = {B\"urgisser, Peter and Clausen, Michael and Shokrollahi,
              Amin},
     TITLE = {Algebraic complexity theory},
    SERIES = {Grundlehren der mathematischen Wissenschaften [Fundamental
              Principles of Mathematical Sciences]},
    VOLUME = {315},
      NOTE = {With the collaboration of Thomas Lickteig},
 PUBLISHER = {Springer-Verlag, Berlin},
      YEAR = {1997},
     PAGES = {xxiv+618},
      ISBN = {3-540-60582-7},
   MRCLASS = {68-02 (12Y05 65Y20 68Q05 68Q15 68Q25 68Q40)},
  MRNUMBER = {1440179},
MRREVIEWER = {Alexander\ I.\ Barvinok},
       DOI = {10.1007/978-3-662-03338-8},
       URL = {https://doi.org/10.1007/978-3-662-03338-8},
}

@article {BL02,
    AUTHOR = {Bosch, Siegfried and Lorenzini, Dino},
     TITLE = {Grothendieck's pairing on component groups of {J}acobians},
   JOURNAL = {Invent. Math.},
  FJOURNAL = {Inventiones Mathematicae},
    VOLUME = {148},
      YEAR = {2002},
    NUMBER = {2},
     PAGES = {353--396},
      ISSN = {0020-9910,1432-1297},
   MRCLASS = {14K15 (14H25 14H40)},
  MRNUMBER = {1906153},
MRREVIEWER = {Alessandra\ Bertapelle},
       DOI = {10.1007/s002220100195},
       URL = {https://doi.org/10.1007/s002220100195},
}

@phdthesis{Storjohann2000,
  author       = {Arne Storjohann},
  title        = {Algorithms for Matrix Canonical Forms},
  school       = {ETH Zurich},
  year         = {2000},
  url          = {https://cs.uwaterloo.ca/~astorjoh/diss2up.pdf}
}

@book {CP18,
    AUTHOR = {Corry, Scott and Perkinson, David},
     TITLE = {Divisors and sandpiles: An introduction to chip-firing},
      NOTE = {},
 PUBLISHER = {American Mathematical Society, Providence, RI},
      YEAR = {2018},
     PAGES = {xiv+325},
      ISBN = {978-1-4704-4218-7},
   MRCLASS = {05-01 (05C25 05C50)},
  MRNUMBER = {3793659},
MRREVIEWER = {Carlos\ Alejandro\ Alfaro},
       DOI = {10.1090/mbk/114},
       URL = {https://doi.org/10.1090/mbk/114},
}

@book {Klivans19,
    AUTHOR = {Klivans, Caroline J.},
     TITLE = {The mathematics of chip-firing},
    SERIES = {Discrete Mathematics and its Applications (Boca Raton)},
 PUBLISHER = {CRC Press, Boca Raton, FL},
      YEAR = {2019},
     PAGES = {xii+295},
      ISBN = {978-1-138-63409-1},
   MRCLASS = {05-01 (05C90 37E25)},
  MRNUMBER = {3889995},
MRREVIEWER = {Paul\ Andrew\ Dreyer, Jr.},
}

@book {CFGRD06,
     TITLE = {Handbook of elliptic and hyperelliptic curve cryptography},
    SERIES = {Discrete Mathematics and its Applications (Boca Raton)},
    EDITOR = {Cohen, Henri and Frey, Gerhard and Avanzi, Roberto and Doche,
              Christophe and Lange, Tanja and Nguyen, Kim and Vercauteren,
              Frederik},
 PUBLISHER = {Chapman \& Hall/CRC, Boca Raton, FL},
      YEAR = {2006},
     PAGES = {xxxiv+808},
      ISBN = {978-1-58488-518-4; 1-58488-518-1},
   MRCLASS = {14G50 (11G05 11G07 11T71 94-00 94A60)},
  MRNUMBER = {2162716},
MRREVIEWER = {Steven\ D.\ Galbraith},
}

@article {Hodges25,
    AUTHOR = {Hodges, Eliot},
     TITLE = {Cokernels of Random {H}ermitian Matrices with Quadratic Integer Entries},
   JOURNAL = {forthcoming}
}

@article {Lee23,
    AUTHOR = {Lee, Jungin},
     TITLE = {Universality of the cokernels of random {$p$}-adic {H}ermitian
              matrices},
   JOURNAL = {Trans. Amer. Math. Soc.},
  FJOURNAL = {Transactions of the American Mathematical Society},
    VOLUME = {376},
      YEAR = {2023},
    NUMBER = {12},
     PAGES = {8699--8732},
      ISSN = {0002-9947,1088-6850},
   MRCLASS = {15B52 (11E39 60B20)},
  MRNUMBER = {4669308},
MRREVIEWER = {Yifeng\ Huang},
       DOI = {10.1090/tran/9031},
       URL = {https://doi.org/10.1090/tran/9031},
}

@incollection {Wood23,
    AUTHOR = {Wood, Melanie Matchett},
     TITLE = {Probability theory for random groups arising in number theory},
 BOOKTITLE = {I{CM}---{I}nternational {C}ongress of {M}athematicians. {V}ol.
              6. {S}ections 12--14},
     PAGES = {4476--4508},
 PUBLISHER = {EMS Press, Berlin},
      YEAR = {[2023] \copyright 2023},
      ISBN = {978-3-98547-064-8; 978-3-98547-564-3; 978-3-98547-058-7},
   MRCLASS = {60E05 (11N45 11R29 11R45 15B52 60B99 60F05)},
  MRNUMBER = {4680411},
MRREVIEWER = {Michael\ Voit},
}

@article {Yan23,
    AUTHOR = {Yan, Eric},
     TITLE = {Universality for {C}okernels of {D}edekind {D}omain {V}alued
              {R}andom {M}atrices},
   JOURNAL = {Michigan Math. J.},
  FJOURNAL = {Michigan Mathematical Journal},
    VOLUME = {75},
      YEAR = {2025},
    NUMBER = {5},
    PAGES = {1071–-1084},
      ISSN = {0026-2285,1945-2365},
   MRCLASS = {11R45 (11K99 13F05 60B20)},
  MRNUMBER = {4979926},
       DOI = {10.1307/mmj/20236348},
       URL = {https://doi-org.libproxy.mit.edu/10.1307/mmj/20236348},
}

@article {Wood15,
    AUTHOR = {Wood, Melanie Matchett},
     TITLE = {Random integral matrices and the {C}ohen-{L}enstra heuristics},
   JOURNAL = {Amer. J. Math.},
  FJOURNAL = {American Journal of Mathematics},
    VOLUME = {141},
      YEAR = {2019},
    NUMBER = {2},
     PAGES = {383--398},
      ISSN = {0002-9327,1080-6377},
   MRCLASS = {11M50 (60B20)},
  MRNUMBER = {3928040},
MRREVIEWER = {Adam\ J.\ Harper},
       DOI = {10.1353/ajm.2019.0008},
       URL = {https://doi-org.libproxy.mit.edu/10.1353/ajm.2019.0008},
}
\bibliographystyle{alpha}

\end{document}